\documentclass[12pt,leqno]{amsart}
\usepackage{amsmath}
\usepackage{amsthm}
\allowdisplaybreaks
\usepackage{amssymb,color}

\def\underset#1#2{{\mathrel{\mathop {{}_{} {#2}}\limits_{{#1}_{}}}}}
\def\upplim_#1{\underset{#1}{\overline\lim}\;}
\def\lowlim_#1{\underset{#1}{\underline\lim}\;}
\newtheorem{lemma}[equation]{Lemma}

\newtheorem{proposition}[equation]{Proposition}

\newtheorem{theorem}[equation]{Theorem}

\newcommand{\C}{{\mathbb{C}}}

\newcommand{\Z}{\mathbb{Z}}

\numberwithin{equation}{section}

\title[Characteristic functions of two meromorphic functions]{Characteristic functions of two meromorphic functions weakly sharing three small\\ functions with bi-weights} 

\author{\sc Si Duc Quang$^{1.2}$}
\author{\sc Phung Nguyen Ngoc Anh$^1$}

\address{$^1$Department of Mathematics, School of Mathematics and Computer Science, Hanoi National University of Education, 136-Xuan Thuy - Cau Giay - Hanoi, Vietnam}
\address{$^2$Institute of Natural Sciences, Hanoi National University of Education, 136-Xuan Thuy - Cau Giay - Hanoi, Vietnam}
\email{quangsd@hnue.edu.vn;anhphungnguyenngoc@gmail.com}

\begin{document}

\begin{abstract} Two meromorphic functions $f$ and $g$ are said to weakly share a small function $a$ with bi-weight $(n,k)$ if the functions $f-a$ and $g-a$ have the same zeros with multiplicities truncated at level $n+1$, while zeros whose multiplicities exceed $k$ are disregarded. In this article, we show that if two meromorphic functions $f$ and $g$ weakly share three small functions $a_i\ (1\le i\le 3)$ with bi-weights $(n_i,k)$ satisfying $n_1n_2n_3>n_1+n_2+n_3+2$ then 
$$(1-\epsilon-\delta_\epsilon)T(r, f)\le (2+\epsilon +\delta_\epsilon )T(r, g)+S(r, g)$$
for every positive number $\epsilon$, where $\delta_\epsilon$ is explicitly estimated depending only on $\epsilon$ and $k$, so that $\delta_\epsilon$ tends to zero as $k$ tends to $+\infty$.
\end{abstract}
\maketitle

\def\thefootnote{\empty}
\footnotetext{
2010 Mathematics Subject Classification:
Primary 30D35; Secondary 32H30, 32A22.\\
\hskip8pt Key words and phrases: meromorphic function, small function, M\"{o}bius transformation.}

\maketitle

\section{Introduction}
Nevanlinna \cite{N} established the theory now bearing his name for meromorphic functions on the complex plane $\C$ in 1926. His theory relates the characteristic function of a meromorphic function to the counting functions of its zeros, poles, and other value distributions. The growth of the characteristic function reflects the behavior of the function near the infinity point. Since then, many authors have developed Nevanlinna theory further and applied it to various problems concerning the growth and value distribution of meromorphic functions. One particularly interesting direction is the study of the relationship between the characteristic functions of meromorphic functions that share certain values or small functions. To state some results in this direction, we first recall the following notation.

Let $f$ be a meromorphic function on $\C$. We denote by $T(r,f)$ the characteristic function of $f$ (see Section 2 for the definition), and by $S(r,f)$ any quantity satisfying
$$\|\ S(r,f)=o(T(r,f)).$$
Here, the notation ``$\|P$'' means that the statement $P$ holds for all $r\in(0,+\infty)$ outside a Borel subset of finite Lebesgue measure. A meromorphic function $a$ is called a small function with respect to $f$ if
$$\|\ T(r,a)=S(r,f).$$

Let $k$ be a positive integer or $+\infty$, and let $S$ be a subset of $\overline{\C}=\C\cup\{\infty\}$. We denote by $E_k^S(a,f)$ the set of all $a$-points of $f$ outside $S$, where an $a$-point of multiplicity $n$ is counted $n$ times if $n\le k$, and counted $k+1$ times if $n>k$. Moreover, if $S=\emptyset$ (resp. $k=+\infty$), we write $E_k(a,f)$ (resp. $E^S(a,f)$) instead of $E_k^S(a,f)$.

Let $f$ and $g$ be two meromorphic functions. We say that $f$ and $g$ weakly weighted share the value $a$ with level $k$ if there exists a discrete subset $S\subset\overline{\C}$ satisfying
$$\|\ N(r,S)=o(T(r,f)+T(r,g)),$$
such that
$$E_k^{S}(a,f)=E_k^{S}(a,g).$$
In the special case where $S=\emptyset$, we simply say that $f$ and $g$ weighted share the value $a$ with level $k$.

When $k=\infty$ (resp. $k=1$), we say that $f$ and $g$ weakly share the value $a$ counting multiplicity (CM) (resp. ignoring multiplicity (IM)).

In 1995, E. Mues \cite{EM} conjectured that if two meromorphic functions $f$ and $g$ share three distinct values counted with multiplicity then
\begin{align*}
    \dfrac{1}{2}T(r,f)\le T(r,g)\le 2T(r,f)+o(1)
\end{align*}
for every $r \in [1,+\infty)$ outside a finite linear measure set, and the numbers $\dfrac{1}{2}, 2$ are sharp.

In 1998, P. Li and C. C. Yang \cite{LY3} established this conjecture. They proved that if two non-constant meromorphic functions $f$ and $g$ share $0,1,\infty$ CM. Then for every positive $\epsilon$, we have the following
\begin{align*}
    T(r,g)\le (2+\epsilon)T(r,f) +S(r,f).
\end{align*}

Later, many related results were obtained for two meromorphic functions that weighted share or weakly weighted share few values (see \cite{AY,QH,LY,Y}). These results differ in their assumptions on the weights and in the applications derived from them. In particular, in 2005, I. Lahiri \cite{L} generalized the previous results by introducing weaker conditions on the weights, and proved the following theorem.

\vskip0.2cm
\noindent
{\bf Theorem A.} (see \cite{L}) \textit{Let $f$ and $g$ be two non-constant meromorphic functions that weighted share $0,1,\infty$ with levels $1,m,k$ respectively, so that $mk-m-k-3>0$. Then
\begin{align*}
     T(r, g) \le 2T(r,f)+S(r,f) \text{ and } T(r,g) \le 2T(r,f)+S(r,g).
\end{align*}}
\indent
Recently, Quang and Ky considered a further weakening of the sharing condition. They introduced the notion of two meromorphic functions weakly sub-weighted sharing a value, defined as follows.

Let $f$ be a meromorphic function, $a$ a value in $\C$ and $k$ a non-negative integer or infinity. Denote by $E^S_{k)}(a, f)$ the set of all $a$-points of $f$ outside a discrete subset $S$ of $\C$, where an $a$-point of multiplicity $m$ is counted $m$ times if $m \le k$ and is omitted if $m >k$ (i.e., not counted). Two meromorphic functions $f$ and $g$ are said to weakly sub-weighted share the value $a$ with level $k$ if there exists a discrete subset $S$ of $\overline\C$ of the counting function equal to $o(T(r,f)+T(r,g))$ such that $E_{k)}^{S}(a,f)=E_{k)}^{S}(a,g).$

Quang and Ky proved the following theorem.

\vskip0.2cm
\noindent
{\bf Theorem B.} (see \cite{QK}) \textit{Let $f, g$ be two non-constant meromorphic functions. Assume that $f$ and $g$ weakly sub-weighted share three distinct values $a_1,a_2,a_3\in\C$ with levels $k_1,k_2,k_3$ respectively, where $k_1, k_2, k_3$ are positive integer or maybe infinity. Then
$$(1-\epsilon-\delta_\epsilon)T(r, f)\le (2+\epsilon +\delta_\epsilon )T(r, g)+S(r, g)$$
for every positive number $\epsilon$, where $\delta_\epsilon=(2^{4/\epsilon^2}+16/\epsilon^4)\left(\dfrac{1}{k_1+1}+\dfrac{1}{k_2+1}+\dfrac{1}{k_3+1}\right)$.}

\vskip0.2cm
Our purpose of this paper is to extend Theorems A and B by considering the case where two functions weakly share three distinct small functions with bi-weights. Now, we recall the following.

Let $f$ be a non-constant meromorphic function and $a$ be a meromorphic function. For $S$ to be a discrete subset of $\overline\C$ and $n,k$ two positive integers. Denote by $E^{S}_{n,k)}(a,f)$ the set of all zeros of $f-a$ outside $S$ with multiplicity not exceeding $k$, where each zero of multiplicity $m\le k$ is counted $m$ times if $m\le n$ and $n+1$ times if $m>n$ ,and all zeros of multiplicity $m>k$ are disregarded. Note that 
$$E^S_{n,+\infty)}(a,f)=E^S_{n}(a,f).$$ 
We say that $f$ and $g$ weakly share the small functions $a$ with the bi-weight $(n,k)$ if there exists a discrete subset $S$ of $\overline\C$ of the counting function equal to $o(T(r,f)+T(r,g))$ such that $E^S_{n,k)}(a,f)=E^S_{n,k)}(a,g)$. 

In this paper, we will prove the following theorem.

\begin{theorem}\label{1.1}
    \textit{Let $f$ and $g$ be two non-constant meromorphic functions and $a_1,a_2,a_3$ three small functions with respect to $f$ and $g$. Let $n_1,n_2,n_3,k$ be four positive integers ($k$ may be $+\infty$) and $\Delta :=n_1n_2n_3-n_1-n_2-n_3-2>0$. If $f$ and $g$ weakly share $a_i\ (1\le i\le 3)$ with bi-weights $(n_i,k)$ then
$$(1-\epsilon-\delta_\epsilon)T(r, f)\le (2+\epsilon +\delta_\epsilon )T(r, g)+S(r, g)$$
for every $\epsilon>0$, where $\displaystyle\delta_\epsilon=(2^{4/\epsilon^2}+16/\epsilon^4)\left(\frac{2\sum_{1\le i<j\le 3}(n_i+1)(n_j+1)}{\Delta(k+1)}+\frac{3}{k+1}\right)$.}
\end{theorem}

Remark.(1) Let $n_1=1$ and let $k=+\infty$ then Theorem \ref{1.1} recovers Theorem A.

(2) Let $n_1=n_2=n_3=n$ and let $n\longrightarrow\infty$ then we get the Theorem B for the case $k_1=k_2=k_3=k$.

\section{Some lemmas and auxiliary results from Nevanlinna theory}

For a divisor $\nu$ on $\C$, we consider $\nu$ as a function on $\C$ with values in $\Z$ so that the set $\overline{\{z|\nu(z)\ne 0\}}$ is a discrete set. We define the counting function of $\nu$ by
$$N(r,\nu)=\int\limits_1^r \dfrac {n(t)}{t}dt \quad (1<r<\infty), \text{ where } n(t)=\sum\limits_{|z|\leq t}\nu (z).$$
For two positive integers $k,M$ (maybe $M= \infty$), we set 
$$ \nu^{[M]}_{\leq k} (z)=
\begin{cases}
\min\{M,\nu (z)\}&\text{ if }\nu (z)\leq k\\
0&\text{ for otherwise. }
\end{cases}$$
and write $N^{[M]}_{k)}(r,\nu)$ for $N(r,\nu^{[M]}_{\leq k})$. We neglect character $^{[M]}$ (resp. $\leq k$) if $M=+\infty$ (resp. $k=+\infty$). In the same way, we define $\nu^{[M]}_{\ge k}$ and write $N^{[M]}_{(k}(r,\nu)$ and $N^{[M]}_{(\ell,k)}(r,\nu)$ for $N(r,\nu^{[M]}_{\ge k})$ and $N(r,(\nu_{\ge\ell})_{\le k}^{[M]})$, respectively. We also write $\overline N(r,\nu)$ for $N^{[1]}(r,\nu)$.

Let $f$ be a non-zero holomorphic function. For each $z_0\in\C$, expanding $f$ as
$ f(z) =\sum_{i=0}^\infty b_i(z-z_0)^i$ around $z_0$, then we define $\nu^{0}_{f}(z_0):=\min\{i\ :\ b_i \ne 0\}$. 

Let $\varphi$ be a non-constant meromorphic function. Then there are two holomorphic function $\varphi_1,\varphi_2$ without common zeros such that $\varphi =\dfrac{\varphi_1}{\varphi_2}$. We define 
$ \nu_\varphi^0:=\nu_{\varphi_1}^0$ and $\nu_\varphi^\infty:=\nu_{\varphi_2}^0$ and $\nu_\varphi=\nu^0_\varphi -\nu^\infty_\varphi$. The proximity function of $\varphi$ is defined by:
$$ m(r,\varphi ):=\frac{1}{2\pi}\int\limits_{0}^{2\pi}\log^+|\varphi (re^{i\theta})|d\theta\ \ (r>1), $$
here $\log^+x=\max\{1,\log x\}$ for $x\in (0,\infty )$. The Nevanlinna characteristic function of $\varphi$ is defined by
$$ T(r,\varphi ):=m(r,\varphi ) +N(r,\nu^\infty_{\varphi}).$$

\begin{theorem}[\cite{Y}, Corollary 1]\label{2.1}
Let $f$ be a non-constant meromorphic function on $\C$. Let $a_1,\dots ,a_q\ (q\ge 3)$ be $q$ distinct small meromorphic functions (with respect to $f$) on $\C$. Then, for each $\epsilon >0$, the
following holds
$$\|\ (q-2-\epsilon)T(r,f)\le \sum_{i=1}^q\overline N(r,\nu^0_{f-a_i})+o(T(r,f)).$$
\end{theorem}

\begin{proposition}[\cite{QA}, Proposition 3.8]\label{2.2}
Let $f$ and $g$ be non-constant meromorphic functions and $a_i,b_i\ (i=1,2,3)$ $ (a_i\ne a_j, b_i\ne b_j, i\ne j)$ be small functions (with respect to $f$ and $g$). Assume that $f$ is not a quasi-M\"{o}bius transformation of $g$. Then for every positive integer $n$ we have the following inequality
$$\|\  N(r,\nu)\le \overline N(r,|\nu^0_{f-a_1}-\nu^0_{g-b_1}|)+\overline N(r,|\nu^0_{f-a_2}-\nu^0_{g-b_2}|)+S(r), $$
where $S(r)=o(T(r,f)+T(r,g))$ outside a finite Borel measure set of $[1,+\infty)$ and $\nu$ is the divisor defined by 
$\nu (z)=\max\{0,\min\{\nu^0_{f-a_3}(z),\nu^0_{g-b_3}(z)\}-1\}.$
\end{proposition}

\begin{lemma}[{Cf. \cite[Lemma 3.2]{QK}}]\label{2.5}
Let $f_1$ and $f_2$ be two non-constant meromorphic functions. If $(f_1^sf_2^t-1)$ is not identically zero for all integers $s$ and $t$ $(|s|+|t|>0)$ then for any positive integer $n$, one of the following assertion holds:\\
(i) $\overline N(r,1;f_1,f_2)\leq (2^{n(n+2)}-n(n+2)-1)\sum_{i=1,2}\left(\overline N(r,\nu^0_{f_i})+\overline N(r,\nu^\infty_{f_i})\right)+\dfrac{1}{n+2}T(r)+S(r);$\\
(ii) $\overline N(r,1;f_1,f_2)\leq (2^{(n+1)^2}+n^4+4n^3+n^2-6n-2)\sum_{i=1,2}\left(\overline N(r,\nu^0_{f_i})+\overline N(r,\nu^\infty_{f_i})\right)+S(r),$\\
where $\overline N(r,1;f_1,f_2)$ denotes the reduced counting function of $f_1$ and $f_2$ related to the common $1$-points, $T(r)=T(r,f_1)+T(r,f_1)$ and $S(r)=o(T(r))$.
\end{lemma}

\begin{lemma}\label{2.6}
Let $f$ and $g$ be non-constant meromorphic functions and $a_i\ (i=1,2,3)$ $ (a_i\ne a_j, i\ne j)$ small functions (with respect to $f$ and $g$). Let $n_1,n_2,n_3,k$ be positive integers ($k$ may be $+\infty$) and $\Delta=n_1n_2n_3-n_1-n_2-n_3-2>0$. Suppose that $f$ and $g$ weakly share three small functions $a_i\ (1\le i\le 3)$ with the bi-weight $(n_i,k)$. If $f$ is not a quasi-M\"{o}bius transformation of $g$ then
$$\|\ \Delta\overline N_{(n_u+1,k)}(r,\nu^{a_u}_f)\le \frac{2(n_v+1)(n_w+1)}{k+1}T(r)+S(r),$$
for every permutation $(u,v,w)$ of $\{1,2,3\}$, where $T(r)=T(r,f)+T(r,g)$, $S(r)=o(T(r))$.
\begin{proof}
Denote by $S$ the discrete subset of $\overline\C$ with $\|\ N(r,S)=o(T(r,f)+T(r,g))$ such that $E^S_{n_i,k)}(a_i,f)=E^S_{n_i,k)}(a_i,g)\ (i=1,2,3).$ By using a quasi-M\"{o}bius transformation if necessary, we may assume that $a_1=0,a_2=\infty,a_3=1$. Without loss of generality, we prove the lemma in the case of $u=1$. We set
\begin{align*}
H=\dfrac{f'}{f-1}-\dfrac{g'}{g-1}
\end{align*}
Since $f$ is not a M\"{o}bius transformation of $g$, $H\not\equiv 0$. Let $z\ (z\not\in S)$ be a zero of $f$ with $\nu^0_f(z)\le k$. If $\nu^0_f(z)> n_1$ then $z$ must be a zero of $H$ with multiplicity at least $n_1$. Therefore, we have
\begin{align}\label{3.6}
\begin{split}
n_1\overline N_{(n_1+1,k)}(r,\nu^{a_1}_f)&\le N(r,\nu^0_{H})\le T(r,H)=N(r,\nu^{\infty}_{H})+S(r)\\
&\leq \overline N(r,|\nu^{\infty}_f-\nu^{\infty}_g |)+\overline N(r,|\nu^{1}_f-\nu^{1}_g |)+S(r)\\
&\leq \sum_{i=2,3}(\overline N_{(n_i+1,k)}(r,\nu^{a_i}_{f})+\sum_{h=f,g}\overline N_{(k+1}(r,\nu^{a_i}_{h}))+S(r)\\
&\leq \sum_{i=2,3}\overline N_{(n_i+1,k)}(r,\nu^{a_i}_{f})+\dfrac{2}{k+1}T(r)+S(r).
\end{split}
\end{align}
Here, the third inequality comes from Proposition \ref{2.2}. Because $a_1,a_2,a_3$ play the same role, we have
$$n_2\overline N_{(n_2+1,k)}(r,\nu^{a_2}_f)\leq\sum_{i=1,3} \overline N_{(n_i+1,k)}(r,\nu^{a_i}_{f})+\dfrac{2}{k+1}T(r)+S(r).$$
Combining (\ref{3.6}) and the above inequality, we get
\begin{align*}
n_1n_2\overline N_{(n_1+1,k)}(r,\nu^{a_1}_f)\le&(n_2+1)\overline N_{(n_3+1,k)}(r,\nu^{a_3}_{f})\\
&+\overline N_{(n_1+1,k)}(r,\nu^{a_1}_{f})+\dfrac{2n_2+2}{k+1}T(r)+S(r).
\end{align*}
Thus,
\begin{align*}
(n_1n_2-1)\overline N_{(n_1+1,k)}(r,\nu^{a_1}_f)\le (n_2+1)\left(\overline N_{(n_3+1,k)}(r,\nu^{a_3}_{f})+\dfrac{2}{k+1}T(r)\right)+S(r).
\end{align*}
Similarly, we have
\begin{align*}
(n_3n_2-1)\overline N_{(n_3+1,k)}(r,\nu^{a_3}_f))\le (n_2+1)\left(\overline N_{(n_1+1,k)}(r,\nu^{a_1}_{f})+\dfrac{2}{k+1}T(r)\right)+S(r).
\end{align*}
From above two inequalities, we obtain
\begin{align*}
(n_3n_2-1)(n_1n_2-1)&\overline N_{(n_1+1,k)}(r,\nu^{a_1}_f)\le (n_3n_2-1)\dfrac{2(n_2+1)}{k+1}T(r)\\ 
&+(n_2+1)^2\left(\overline N_{(n_1+1,k)}(r,\nu^{a_1}_{f})+\dfrac{2}{k+1}T(r)\right)+S(r).
\end{align*}
This implies that
$$\Delta\overline N_{(n_1+1,k)}(r,\nu^{a_1}_f)\le \frac{2(n_2+1)(n_3+1)}{k+1}T(r)+S(r).$$
The lemma is proved.
\end{proof}
\end{lemma}

\section{Proof of Theorem 1}
We first prove the foloowing lemma.
\begin{lemma}\label{3.1} 
    \textit{Let $f$ and $g$ be non-constant meromorphic functions, $f$ is not a M\"{o}bius transformation of $g$. Assume that $f$ and $g$ weakly share three values $a_1=0, a_2=1, a_3=\infty$ with the bi-weights $(n_i,k)$ respectively, where $n_1,n_2,n_3$ and $k$ be positive integers or $+\infty$ with $n_i\le k\ (1\le i\le 3)$ satisfying $\Delta:=n_1n_2n_3-n_1-n_2-n_3-2>0$. We have one of the followings holds:
\begin{itemize}
    \item [(a)] $T(r,f)\le T(r,g)+\left(\displaystyle\frac{8\sum_{1\le i<j\le 3}(n_i+1)(n_j+1)}{\Delta(k+1)}+\frac{5}{k+1}\right)T(r) +S(r)$.
    \item [(b)]$
    T(r, f)+T(r, g)\le N(r,g)+N\left(r, \dfrac{1}{g-1}\right) +N\left(r, \dfrac{1}{g}\right)+N_0(r)\\
    +\displaystyle\left(\frac{116\sum_{1\le i<j\le 3}(n_i+1)(n_j+1)}{\Delta(k+1)}+\frac{78}{k+1}\right)T(r)+ S(r).$
\end{itemize}}
\end{lemma}

\begin{proof}. 
We define $h_1:=\dfrac{f}{g}$ and $h_2:=\dfrac{f-1}{g-1}$. It easy to see that $T(r, h_i)\le T(r)$, $i=1,2$. Because $f$ is not a M\"{o}bius transformation of $g$, $\beta =\dfrac{h_1'}{h_1}-\dfrac{h_2'}{h_2} \not\equiv 0$. Set $\alpha=-\dfrac{1}{\beta}\dfrac{h_2'}{h_2}$. From Lemma \ref{2.6}, we have
\begin{align}\label{3.2}
\begin{split}
    N(r, \beta)&\le N(r, \nu_1)\le \sum_{h=f, g} \sum_{i=1}^{3}\left[\overline{N}_{(n_i+1,k)}(r, \nu_h^{a_i})+\overline{N}_{(k+1}(r, \nu_h^{a_i})\right]\\
    &\le \left(\frac{4\sum_{1\le i<j\le 3}(n_i+1)(n_j+1)}{\Delta(k+1)}+\frac{3}{k+1}\right)T(r)+S(r),\\
\end{split} 
\end{align} 
\begin{align}\label{3.3}\begin{split}
    T(r, \alpha)&\le T\left(r, \dfrac{h_2'}{h_2}\right)+T(r, \beta)=N\left(r, \dfrac{h_2'}{h_2}\right)+N(r, \beta)+S(r)\\
    & \le \sum_{h=f, g} \sum_{i=2,3}\left(\overline{N}_{(n_i+1,k)}(r, \nu_h^{a_i}) +\overline{N}_{(k+1}(r, \nu_h^{a_i})\right)+ N(r,\beta) +S(r)\\
    & \le \left(\frac{8\sum_{1\le i<j\le 3}(n_i+1)(n_j+1)}{\Delta(k+1)}+\frac{5}{k+1}\right)T(r) +S(r).
\end{split}\end{align}
where $\nu_1$ is the reduced divisor counting all points $z$ such that 
$$\nu_f(z)\ne \nu_g(z)\text{ or }\nu_{f-1}(z)\ne \nu_{g-1}(z).$$

It is easy to compute that
\begin{align*}
    f = \dfrac{h_1-h_1h_2}{h_1-h_2}=\dfrac{h_2^{-1}-1}{h_2^{-1}-h_1^{-1}}
\end{align*}
Thus 
\begin{align*}
    f-\alpha=\dfrac{(1-\alpha)h_2^-1+\alpha h_1^{-1}-1}{h_2^{-1}-h_1^{-1}}
\end{align*}
Now, we set $F_1=(\alpha-1)h_2^{-1}-\alpha h_1^{-1}+1, F_2=(1-\alpha)h_2^{-1}, F_3=\alpha h_1^{-1}$. Then
$$F_1+F_2+F_3=1$$
and
\begin{align*}
    f-\alpha=\dfrac{F_1}{h_2^{-1}-h_1^{-1}}
\end{align*}
By using Claim 3.11 in \cite{QK}, we have
\begin{align}\label{3.4}
\begin{split}
 N\left(r,\dfrac{1}{f-\alpha}\right)&-N(r,f-\alpha)\\
&\ge N\left(r,\dfrac{1}{F_1}\right)-N\left(r,\dfrac{1}{h_2^{-1}-h_1^{-1}}\right)-N(r,\alpha)\\
&\ge N\left(r,\dfrac{1}{F_1}\right)-N\left(r,\dfrac{1}{h_2^{-1}-h_1^{-1}}\right)+N(r,f)-2T(r,\alpha).
\end{split}
\end{align}
\noindent We consider the following two cases:

\textbf{Case 1:} $F_1,F_2,F_3$ are linearly dependent. Then there exist constants $c_1,c_2,c_3$, not all zeros, such that
\begin{align*}
    c_1F_1+c_2F_2+c_3F_3=0
\end{align*}
\begin{itemize}
    \item If $c_1 =0,$ then $c_2F_2+c_3F_3=0$. This implies that 
		$$\frac{f-1}{f}=-\frac{c_2}{c_3}\frac{1-\alpha}{\alpha}\frac{g-1}{g}.$$
		Therefore
		$$T(r,f)\le T(r,g)+T(r,\alpha)+O(1).$$
	\item If $c_1\neq 0$ then $c_2'h_2+c_3'h_3\equiv 1$, where $\displaystyle c_2'=(1-\frac{c_2}{c_1}),c_3'=(1-\frac{c_3}{c_1}).$
		It is easy to show that
		$$f(f-1)=c_2'(1-\alpha)(g-1)f+c_3'\alpha(f-1).$$
		Thus
		$$2T(r,f)\le T(r,g)+T(r,f)+T(r,\alpha)+O(1),$$
		\text{ i.e.,}$$\ T(r,f)\le T(r,g)+T(r,\alpha)+O(1).$$
\end{itemize}     
Combining this inequality with the inequality (\ref{3.3}), we obtain the desired inequality (a) of the lemma.

\textbf{Case 2:} $F_1,F_2,F_3$ are linearly independent. By inequality (1.1.18) in \cite[Lemma 1.102]{PPC}, we have
 \begin{align*}
			T(r,F_1)\le N\Big(r,\frac{1}{F_1}\Big)+2\sum_{i=1,2}\left(\overline{N}(r,F_i)+\overline{N}(r,\frac{1}{F_i})\right)+6T(r,\alpha)+S(r).
\end{align*}
Combining this with inequality (\ref{3.3}), thus
\begin{align*}
			N\Big(r,\dfrac{1}{f-\alpha}\Big)&\ge T(r,(h_1)-N\Big(r,\dfrac{1}{h_2^{-1}-h_1^{-1}}\Big)+N(r,f)\\
			&-2\sum_{i=1,2}\left(\overline{N}(r,h_i)+\overline{N}(r,\frac{1}{h_i})\right)-8T(r,\alpha)+S(r).
\end{align*}	

We define $F:=\dfrac{\alpha-1}{h_2}-\dfrac{\alpha}{h_1}$ and $G:=\dfrac{1}{h_2}-\dfrac{1}{h_1}$. Then we have 
\begin{align}\label{3.13}
F'=\Big(\alpha'-(\alpha-1)\dfrac{F_2'}{F_2}\Big)\dfrac{1}{F_2}-\Big(\alpha'-\alpha\dfrac{F_1'}{F_1}\Big)\dfrac{1}{F_1}.
\end{align}
Let $\gamma=-\alpha'+\alpha(\alpha-1)\Big(\dfrac{h_2'}{h_2}-\dfrac{h_1'}{h_1}\Big).$

$\bullet$  If $\gamma\equiv 0$ then there is a constant $c$ such that
		$\frac{h_2}{h_1}=\frac{c(\alpha-1)}{\alpha}$, i.e, $\frac{f-1}{f}=\frac{c(\alpha-1)}{\alpha}\frac{g-1}{g}$. It yields that
		$$T(r,f)\le T(r,g)+T(r,\alpha).$$
		Combining this inequality with inequality (\ref{3.3}), we again obtain the desired inequality (a) of the lemma.

$\bullet$ If $\gamma \not\equiv 0$, then it follows from (\ref{3.13}) and the definition of $F$ that
$$\frac{1}{F_2}=\frac{1}{\gamma}\left(\Big(\alpha'-\alpha\dfrac{F_1'}{F_1}\Big)F-\alpha F'\right)\text{ and } \frac{1}{F_1}=\frac{1}{\gamma}\left(\Big(\alpha'-(\alpha-1)\dfrac{F_2'}{F_2}\Big)F-(\alpha-1)F'\right).$$
This implies that
$$G=\frac{1}{\gamma}\left((\alpha-1)(\frac{F_2'}{F_2}-\alpha\frac{F_1'}{F_1})F+F'\right).$$
We follows the argument in \cite{QK} to give the following estimate. First, we have
\begin{align*}
T(r,G)&\le T(r,\gamma)+T\left(r,(\alpha-1)(\frac{F_2'}{F_2}-\alpha\frac{F_1'}{F_1})F+F')\right)\\
&\le N(r,\gamma)+m(r,F)+N\left(r,(\alpha-1)(\frac{F_2'}{F_2}-\alpha\frac{F_1'}{F_1})F+F')\right)+S(r)\\
&\le N(r,\gamma)+m(r,F)+N(r,F)+\overline N(r,F)+2N(r,\alpha)+N(r,\nu_1)+S(r)\\
&\le T(r,F)+2N(r,\nu_1)+4N(r,\alpha)+S(r),
\end{align*}
Hence, we obtain
\begin{align*}
N\Big(r,\dfrac{1}{f-\alpha}\Big) &\ge T(r,F_2^{-1}-F_1^{-1})-N\Big(r,\dfrac{1}{F_2^{-1}-F_1^{-1}}\Big)+N(r,f)-2N(r,\nu_1)\\
&-2\sum_{i=1,2}\left (\overline{N}(r,F_i)+\overline N(r,\frac{1}{F_i})\right)-4N(r,\alpha)-8T(r,\alpha)+S(r)\\
&=m\Big(r,\dfrac{1}{F_2^{-1}-F_1^{-1}}\Big)+N(r,f)-2\sum_{i=1,2}\left (\overline{N}(r,F_i)+\overline N(r,\frac{1}{F_i})\right)\\
&-2N(r,\nu_1)-12T(r,\alpha)+S(r).
\end{align*}
From $G=F_2^{-1}-F_1^{-1}$, we have $G'=\lambda_1 F_1^{-1}-\lambda_2 F_2^{-1}$ where $\lambda_1=\dfrac{F_1'}{F_1},\lambda_2=\dfrac{F_2'}{F_2}$. Obviously, $\lambda_1-\lambda_2 \neq 0$. Therefore,
$F_2^{-1}=\frac{\lambda_1 G}{\lambda_1-\lambda_2}+\frac{G'}{\lambda_1-\lambda_2},$
and hence
\begin{align*}
f&=\dfrac{F_2^{-1}-1}{F_2^{-1}-F_1^{-1}}=-\dfrac{1}{F_2^{-1}-F_1^{-1}}+\dfrac{F_2^{-1}}{F_2^{-1}-F_1^{-1}}\\
&=-\dfrac{1}{F_2^{-1}-F_1^{-1}}+\dfrac{\lambda_1}{\lambda_1-\lambda_2}+\dfrac{G'}{(\lambda_1-\lambda_2)G}.
\end{align*}
This yields that
\begin{align*}
m(r,f)&\le m\Big(r,\dfrac{1}{F_2^{-1}-F_1^{-1}}\Big)+m(r,\frac{1}{\lambda_1-\lambda_2})+S(r)\\
&\le m\Big(r,\dfrac{1}{F_2^{-1}-F_1^{-1}}\Big)+T(r,\lambda_1-\lambda_2)+S(r)\\
&\le m\Big(r,\dfrac{1}{F_2^{-1}-F_1^{-1}}\Big)+N(r,\nu_1)+S(r).
\end{align*}
It implies that
\begin{align*}
N\Big(r,\dfrac{1}{f-\alpha}\Big)&\ge T(r,f)-2\sum_{i=1,2}\left (\overline{N}(r,F_i)+\overline N(r,\frac{1}{F_i})\right)\\
&-3N(r,\nu_1)-12T(r,\alpha)+S(r).+S(r).
\end{align*}
Hence 
\begin{align*}
N\Big(r,\dfrac{1}{g'}\Big)+N_0(r)+N(r,\beta)&\ge T(r,f)-2\sum_{i=1,2}\left (\overline{N}(r,F_i)+\overline N(r,\frac{1}{F_i})\right)\\
&-3N(r,\nu_1)-12T(r,\alpha)+S(r).
\end{align*}
From the the second main theorem, we have
\begin{align*}
    T(r,f)+T(r,g)&\le T(r,f)+N(r,g)+N\left(r, \dfrac{1}{g-1}\right) +N\left(r, \dfrac{1}{g}\right)-N\left(r, \dfrac{1}{g'}\right)+S(r)\\
    &\le N(r,g)+N\left(r, \dfrac{1}{g-1}\right) +N\left(r, \dfrac{1}{g}\right)+N_0(r)+N(r, \beta)\\
    &+2\sum_{i=1,2}\left(\overline{N}(r, h_i)+\overline{N}\left(r, \dfrac{1}{h_i}\right)\right)+3N(r, \nu_1)+12T(r, \alpha)+S(r)
\end{align*}
From Lemma \ref{2.6}, we have
\begin{align*}
    \sum_{i=1,2}&\left(\overline{N}(r, h_i)+\overline{N}\left(r, \dfrac{1}{h_i}\right)\right)\le \sum_{i=1}^3\overline N(r,|\nu^{a_i}_f-\nu^{b_i}_g|)\\
    &\le \sum_{i=1}^3\overline N_{(n_i+1,k)}(r,\nu^{a_i}_f)+\sum_{i=1}^3\left(\overline N_{(k+1}(r,\nu^{a_i}_f)+\overline N_{(k+1}(r,\nu^{b_i}_g) \right)+S(r)\\
    &\le \left(\frac{2\sum_{1\le i<j\le 3}(n_i+1)(n_j+1)}{\Delta(k+1)}+\frac{3}{k+1}\right)T(r)+S(r).
\end{align*} 
Combining the above inequalities, we get
\begin{align*}
    T(r, f)+T(r, g)&\le N(r,g)+N\left(r, \dfrac{1}{g-1}\right) +N\left(r, \dfrac{1}{g}\right)+N_0(r)\\
    &\le N(r,g)+N\left(r, \dfrac{1}{g-1}\right) +N\left(r, \dfrac{1}{g}\right)+N_0(r)\\
    &+\left(\frac{116\sum_{1\le i<j\le 3}(n_i+1)(n_j+1)}{\Delta(k+1)}+\frac{78}{k+1}\right)T(r)+ S(r). \\
\end{align*}
We obtain the assertion (b).
\end{proof}

\begin{proof}[Proof of Theorem \ref{1.1}] If $f$ is a quasi-M\"{o}bius transformation of $g$ then 
$$T(r, f)=T(r, g)+S(r),$$ 
and hence we get the conclusion of theorem.Therefore, we may assume that $f$ is not a quasi-M\"{o}bius transformation of $g$. Moreover,without loss of generality, we may assume that $a_1=0, a_2=1, a_3=\infty$. 

With $h_1, h_2$ are defined in Lemma \ref{3.1}, then 
$$T(r, h_i)\le T(r)(i=1,2).$$ 
If there exist integers $s, t$ ($\|s\|+\|t\|>0$) such that $h_1^sh_2^t-1\equiv 0)$ then $f^s(f-1)^t=g^s(g-1)^t$, thus $T(r,f)=T(r,g)$ and we get the conclusion of the theorem.

Therefore, we may suppose that $h_1^sh_2^t-1\not\equiv 0$ for all integers $s, t$ ($\|s\|+\|t\|>0$), then by Lemma \ref{3.1}, one of the fowlloing inequalities holds:
\begin{align}\label{3.5}
    T(r,f)\le T(r,g)+\left(\dfrac{8\sum_{1\le i<j\le 3}(n_i+1)(n_j+1)}{\Delta(k+1)}+\frac{5}{k+1}\right)T(r) +S(r),
\end{align}
\begin{align}\label{3.6}
\begin{split}
    T(r, f)+T(r, g)\le N(r,g)+N\left(r, \dfrac{1}{g-1}\right) +N\left(r, \dfrac{1}{g}\right)+N_0(r)\\
    +\left(\frac{116\sum_{1\le i<j\le 3}(n_i+1)(n_j+1)}{\Delta(k+1)}+\frac{78}{k+1}\right)T(r)+ S(r).
\end{split} 
\end{align}

$\bullet$ If the inequality (\ref{3.5}) holds then we completely get the conclusion of the theorem.

$\bullet$ We suppose that the inequality (\ref{3.6}) holds. From the Lemma \ref{2.5}, for any positive integer $n$, we get
\begin{align*}
    \overline N(r,1;h_1,h_2)&\leq \dfrac{1}{n+2}\left(T(r,h_1)+T(r, h_2)\right) +\lambda\sum_{i=1,2}\left(\overline N(r,\nu^0_{h_i})+\overline N(r,\nu^\infty_{h_i})\right)\\
    & \le \left (\dfrac{2}{n+2}+\lambda\left(\frac{2\sum_{1\le i<j\le 3}(n_i+1)(n_j+1)}{\Delta(k+1)}+\frac{3}{k+1}\right) \right)T(r),
\end{align*}
where $h_1, h_2$ are defined in Lemma \ref{3.1} and 
$$\lambda=2^{(n+1)^2}+n^4+4n^3+n^2-6n-2.$$
From the above inequality and  (\ref{3.6}) we have
\begin{align*}
&T(r,f)+T(r,g)\\&
\le 3T(r,g)+\left (\dfrac{2}{n+2}+\left(\frac{(2\lambda+116)\sum_{1\le i<j\le 3}(n_i+1)(n_j+1)}{\Delta(k+1)}+\frac{3\lambda+78}{k+1}\right) \right)T(r).
\end{align*}
We choose $n$ be the smallest integer  such that $\dfrac{2}{n+2}<\epsilon$, i.e., $n+1\leq \dfrac{2}{\epsilon}$. Hence
$$T(r,f)\le 2T(r,g)+\left (\epsilon+\delta_\epsilon\right) T(r).$$
Thus
$$(1-\epsilon-\delta_\epsilon)T(r, f)\le (2+\epsilon +\delta_\epsilon )T(r, g)+S(r, g).$$
By the symmetry of $f$ and $g$, we also have
$$(1-\epsilon-\delta_\epsilon)T(r, g)\le (2+\epsilon +\delta_\epsilon )T(r, f)+S(r, f).$$
Hence, $S(r,f)=S(r,g)=S(r)$. Therefore, we get the desired inequality. 

The theorem is proved
\end{proof}

\section*{Data availability}
Data sharing not applicable to this article as no datasets were generated or analyzed during the current study.

\section*{Disclosure statement}
No potential conflict of interest was reported by the author(s).

\end{document}